\documentclass{amsart}
\usepackage{amsfonts, amsbsy, amsmath, amssymb}

\usepackage[figuresright]{rotating} 

\allowdisplaybreaks[1]

\newtheorem{thm}{Theorem}[section]
\newtheorem{lem}[thm]{Lemma}
\newtheorem{cor}[thm]{Corollary}

\newtheorem{conj}[thm]{Conjecture}

\newtheorem{rmk}[thm]{Remark}
\newtheorem{fac}[thm]{Fact}

\numberwithin{equation}{section}

\theoremstyle{definition}

\newcommand{\f}{\Bbb F}

\begin{document}

\title{Proof of a conjecture on monomial graphs}

\author[Xiang-dong Hou]{Xiang-dong Hou}
\address{Department of Mathematics and Statistics,
University of South Florida, Tampa, FL 33620}
\email{xhou@usf.edu}

\author{Stephen D. Lappano}
\address{Department of Mathematics and Statistics,
University of South Florida, Tampa, FL 33620}
\email{slappano@mail.usf.edu}

\author{Felix Lazebnik}
\address{Department of Mathematical Sciences, University of Delaware, Newark, DE 19716}
\email{fellaz@udel.edu}



\begin{abstract}
Let $e$ be a positive integer, $p$ be an odd prime, $q=p^{e}$, and
 $\f _q$ be the finite field of $q$ elements. Let $f,g \in \f _q [X,Y]$.
 The graph $G=G_q(f,g)$ is a bipartite graph with vertex partitions
$P=\f_q^3$ and $L=\f_q^3$, and edges defined as follows:
 a vertex $(p)=(p_1,p_2,p_3)\in P$ is adjacent to a vertex $[l] = [l_1,l_2,l_3]\in L$
 if and only if
$p_2 + l_2 = f(p_1,l_1)$  and $p_3 + l_3 = g(p_1,l_1)$.
Motivated by some questions in finite geometry and extremal graph theory,  
Dmytrenko, Lazebnik and Williford conjectured in 2007 that
if $f$ and $g$ are both monomials and  $G$ has no cycle of length
less than eight, then $G$ is isomorphic to the graph $G_q(XY,XY^2)$.  
They proved several instances of the conjecture by reducing it to the property of  
polynomials $A_k= X^k[(X+1)^k - X^k]$ and $B_k= [(X+1)^{2k} - 1] X^{q-1-k} - 2X^{q-1}$
 being permutation polynomials of $\f_q$.

In this paper we prove the conjecture by obtaining new results on the polynomials $A_k$ and $B_k$,
which are also of interest on their own.
\end{abstract}

\maketitle

\section{Introduction}

All graphs considered in this paper are finite, undirected, with no loops or multiple edges.
 All definitions of graph-theoretic terms that we omit can be found  in Bollob\' as \cite{Bollobas98}. 
 The {\it order} of a graph is the number of its vertices.
 The {\it degree} of a vertex of a graph is the number of vertices adjacent to it.
 A graph is called $r$-{\it regular} if degrees of all its vertices are equal to $r$. 
A graph is called {\it connected} if every pair of its distinct
vertices is connected by a path. The {\it distance} between two
distinct vertices in a connected graph is the length of the shortest
path connecting them. The {\it diameter} of a connected graph is the
greatest of all distances between its vertices. The {\it girth} of a graph containing cycles is the length of a shortest cycle.

Though generalized quadrangles are traditionally viewed as incidence geometries (see Payne \cite{Payne90}, Payne and Thas \cite{Payne-Thas2009}, or  Van Maldeghem \cite{vM98}), they can be presented in purely graph theoretic terms,  as we choose to do in this paper.  A {\it generalized quadrangle of order $r$}, $r\ge 1$, denoted by GQ$(r)$,  is a bipartite $(r+1)$-regular graph of diameter four and girth eight.  For every prime power $r$, GQ($r$) exist; no example of GQ($r$) is known when $r$ is not a prime power.  Moreover, when $r=q$ is an odd prime power, up to isomorphism, only one GQ($q$) is known (it corresponds to two dual geometries).  We will denote it by $\Lambda_q$. It is easy to see that if a GQ($r$) exists,  then it has $2(r^3+r^2+r+1)$ vertices and $(r+1)(r^3+r^2+r+1)$ edges.  GQ's are extremal objects for several problems in extremal graph theory. One of them is finding a $k$-regular graph of girth eight and of minimum order, often called a $(k,8)$-cage; see  a survey by  Exoo and Jajcay \cite{ExoJaj13}.  Another is a problem of finding a graph of diameter four,  maximum degree $\Delta$  and of maximum order;  see a survey by Miller and \v{S}ir\'{a}\v{n} on Moore graphs \cite{MilSir13}.  Yet another is to determine the greatest number of edges in a graph of a given order and girth at least eight; see Bondy \cite{Bondy02}, F\"uredi and Simonovits \cite{FurSim13}, and Hoory \cite{Hor02}.

Let $q$ be a prime power, and let $\mathbb F_q$ be the field of $q$ elements. The notion of a permutation polynomial will be central in this paper: A {\em permutation polynomial} (PP) of $\f _q$  is a polynomial $h\in \f _q [X]$ such that the function defined by $a\mapsto h(a)$ is a bijection on $\f _q$. For more information on permutation polynomials, we refer the reader to a recent survey \cite{Hou-FFA-2015} by Hou and the references therein.

Let $f,g \in \f _q [X,Y]$.
 The graph $G=G_q(f,g)$ is a bipartite graph with vertex partitions
$P=\f_q^3$ and $L=\f_q^3$, and edges defined as follows:
 a vertex $(p)=(p_1,p_2,p_3)\in P$ is adjacent to a vertex $[l] = [l_1,l_2,l_3]\in L$
 if and only if
$$p_2 + l_2 = f(p_1,l_1)\;\;\;\text{ and}\;\;\;  p_3 + l_3 = g(p_1,l_1).$$

For the origins, properties and applications of graphs $G_q(f,g)$ and their generalizations, see 
Lazebnik and Ustimenko \cite{LU93}, Lazebnik and Woldar \cite{LWol01} and references therein.  For some later results, see Dmytrenko, Lazebnik and Wiliford \cite{DLW-FFA-2007} and \mbox{Kronenthal} \cite{Kronenthal-FFA-2012}.


If $f$ and $g$  are monomials, we refer to  $G_q(f,g)$ as a {\it
monomial graph}. They all are  $q$-regular and of order $2q^3$. 
It is easy to check that the monomial graph $\Gamma_3(q) = G_q(XY,XY^2)$  has  girth eight. Most importantly, it is also known that if $q$ is odd,  $\Gamma_3(q)$  is isomorphic to an induced subgraph of the graph $\Lambda_q$;  see  \cite{LU93}, Payne \cite{P70}, and  \cite{vM98},  where $\Gamma_3(q)$ is presented with slightly different equations.  The presentation  of $\Gamma_3(q)$ as  $G_{q} (XY,XY^2)$ appears in Viglione \cite{Vig02}.  
The graph $\Gamma_3(q)$  can be obtained by deleting an edge from $\Lambda_q$ together with all vertices  at the distance at most two from the edge. It is known that graph $\Lambda_q$ is edge-transitive, and so the construction of $\Gamma_3(q)$ above does not depend on the edge of $\Lambda_q$.  We can also say that $\Lambda _q$ is obtained from $\Gamma_3(q)$ by ``attaching" to it a $(q+1)$-tree, i.e., a tree with $2q^3$ leaves whose every inner vertex is of degree $q+1$. For  details see Dmytrenko \cite{Dmy04}.

This suggests  to look for graphs $G_q(f,g)$   of girth eight not isomorphic to $\Gamma_3(q)$, where $q$ is odd. If they exist, one may try to attach a $(q+1)$-tree to them and construct a new GQ$(q)$. This idea of constructing a new GQ$(q)$ was suggested by Ustimenko in the 1990's.

 As monomial graphs are in `close vicinity' of the monomial graph $\Gamma_3(q)=G_q(XY, XY^2)$,  it is natural to begin looking for new  graphs $G_q(f,g)$ of girth eight among the monomial graphs.  This motivated papers  \cite{DLW-FFA-2007} and \cite{Kronenthal-FFA-2012}.  Another reason for looking at the monomial graphs first is the following:  For even $q$,  the monomial graphs do lead to a variety of non-isomorphic generalized quadrangles; see Cherowitzo \cite{Che99}, Glynn \cite{Glynn83} and \cite{Glynn89},
Payne \cite{Payne90}, and  \cite{vM98}. It is conjectured in
\cite{Glynn83} that known examples of such quadrangles represent  all possible ones.  The conjecture was checked by computer for all $e\le 28$ in \cite{Glynn89}, and for all $e\le 40$  by Chandler \cite{Chandler05}.

The results \cite{DLW-FFA-2007} and \cite{Kronenthal-FFA-2012} (see more on them in the next section) suggest that for odd $q$,  monomial graphs of girth at least eight are  isomorphic to $\Gamma_3(q)$. 
In fact, the main conjecture of \cite{DLW-FFA-2007} and \cite{Kronenthal-FFA-2012} is the following.

\begin{conj}\label{C1.1}
Let $q$ be an odd prime power. Then every monomial graph of girth eight is isomorphic to $\Gamma_3(q)$.
\end{conj}

In an attempt to prove Conjecture~\ref{C1.1}, two more related conjectures were proposed in \cite{DLW-FFA-2007} and \cite{Kronenthal-FFA-2012}.
For an integer $1\le k\le q-1$, let 
\begin{equation}\label{1.1}
A_k=X^k\bigl[(X+1)^k-X^k\bigr]\in \f_q[X]
\end{equation}
and 
\begin{equation}\label{1.2}
B_k=\bigl[(X+1)^{2k}-1\bigr]X^{q-1-k}-2X^{q-1}\in \f_q[X].
\end{equation}

\medskip

\noindent{\bf Conjecture A.} {\it Let $q$ be a power of an odd prime $p$ and $1\le k\le q-1$. Then $A_k$
is a PP of $\f_q$ if and only if $k$ is a power of $p$.
}

\medskip

\noindent{\bf Conjecture B.} {\it Let $q$ be a power of an odd prime $p$ and $1\le k\le q-1$. Then $B_k$
is a PP of $\f_q$ if and only if $k$ is a power of $p$.
}

\medskip

The logical relation between the above three conjectures is as follows. It was proved in \cite{DLW-FFA-2007} that for odd $q$, every monomial graph of girth $\ge 8$ is isomorphic to $G_q(XY,X^kY^{2k})$, where $1\le k\le q-1$ is an integer not divisible by $p$ for which both $A_k$ and $B_k$ are PPs of $\f_q$. In particular, either of Conjectures A and B implies Conjecture~\ref{C1.1}.

In \cite{DLW-FFA-2007} and \cite{Kronenthal-FFA-2012}, the above conjectures were shown to be true under various additional conditions. The main objective of the present paper is to confirm Conjecture~\ref{C1.1}. This is achieved by making progress on Conjectures A and B. Our results fall short of establishing the claims of Conjectures A and B. However, when considered together, these partial results on Conjectures A and B turn out to be sufficient for proving Conjecture~\ref{C1.1}.

The paper is organized as follows. In Section~2, we review the prior status of the three conjectures and highlight the contributions of the present paper. The permutation property of a polynomial $f\in\f_q[X]$ is encoded in the power sums $\sum_{x\in\f_q}f(x)^s$, $1\le s\le q-1$. In Section~3, we compute the power sums of $A_k$ and $B_k$, from which we derive necessary and sufficient conditions for $A_k$ and $B_k$ to be PPs of $\f_q$. Further results on $A_k$ and $B_k$ are collected in Section~4. The results gathered in Section~4 are a bit more than we need in this paper but can be useful for further work on Conjectures A and B. Sections~5 and 6 deal with Conjectures A and B, respectively. We show that each of them is true under a simple additional condition. Finally, the proof of Conjecture~\ref{C1.1} is given in Section~7. 

Two well known facts about binomial coefficients are frequently used in the paper without further mentioning. Lucas' theorem (see Lucas \cite{Lucas-AJM-1878}) states that for a prime $p$ and integers $0\le m_i, n_i\le p-1$, $0\le i\le e$,
\[
\binom{m_0+m_1p+\cdots+m_ep^e}{n_0+n_1p+\cdots+n_ep^e}\equiv \binom{m_0}{n_0}\cdots\binom{m_e}{n_e}\pmod p.
\]
Consequently, for integers $0\le n\le m$, $\binom mn\equiv 0\pmod p$ if and only if the sum $n+(m-n)$ has at least one carry in base $p$.

Throughout the paper, most equations involving integers should be treated as equations in the characteristic of $\f_q$, i.e., in characteristic $p$.

\section{The Conjectures: Prior Status and New Contributions}

Let $q=p^e$, where $e$ is a positive integer. 
The ``if'' portions of Conjectures~A and B are rather obvious. It is also clear that if Conjectures~A (or B) is true for $k=k_0$, then it is also true for all $k$ in the $p$-cyclotomic coset of $k_0$ modulo $q-1$, i.e., for all $k\equiv p^ik_0\pmod{q-1}$, where $i\ge 0$.
It was proved in \cite{DLW-FFA-2007} that Conjecture~\ref{C1.1} is true for a given $q$ if the $k$'s for which both $A_k$ and $B_k$ are PPs of $\f_q$ are precisely the powers of $p$. In particular, either of Conjectures~A and B implies Conjecture~\ref{C1.1}.

\subsection{Prior Status of Conjecture~\ref{C1.1}}\

For an integer $e>1$, let $\text{gpf}(e)$ denote the greatest prime factor of $e$, and additionally, define $\text{gpf}(1)=1$.

\begin{thm}[{\cite[Theorem~3]{DLW-FFA-2007}}]\label{T1.1}   
Conjecture~\ref{C1.1} is true if one of the following occurs.
\begin{itemize}
  \item [(i)] $q=p^e$, where $p\ge 5$ and $\text{\rm gpf}(e)\le 3$. 
  \item [(ii)] $3\le q\le 10^{10}$.
\end{itemize}
\end{thm}  

The above result was recently extended by Kronenthal \cite{Kronenthal-FFA-2012} as follows.

\begin{thm}[{\cite[Theorem~4]{Kronenthal-FFA-2012}}]\label{T1.2}
For each prime $r$ or $r=1$, there is a positive integer $p_0(r)$ such that Conjecture~\ref{C1.1} is true for $q=p^e$ with $\text{\rm gfp}(e)\le r$ and $p\ge p_0(r)$. In particular, one can choose $p_0(5)=7$, $p_0(7)=11$, $p_0(11)=13$.
\end{thm}

\begin{rmk}\label{R1.3}\rm
\cite[Theorem~3]{DLW-FFA-2007} and the proof of \cite[Theorem~1]{DLW-FFA-2007} allow one to choose $p_0(3)=5$ and $p_0(1)=3$. However, in general, the function $p_0(r)$ given in \cite{Kronenthal-FFA-2012} is not explicit.
\end{rmk}


\subsection{Prior Status of Conjecture~A}\

The proof of \cite[Theorem~1]{DLW-FFA-2007} implies that Conjecture~A is true for $q=p$.


\subsection{Prior Status of Conjecture~B}\

For each odd prime $p$, let $\alpha(p)$ be the smallest positive even integer $a$ such that 
\[
\binom a{a/2}\equiv (-1)^{a/2}2^a\pmod p.
\]
The proof of \cite[Theorem~4]{Kronenthal-FFA-2012} implies the following.

\begin{thm}\label{T1.4}
Let $p$ be an odd prime. If Conjecture~B is true for $q=p^e$, then it is also true for $q=p^{em}$ whenever 
\[
m\le\frac{p-1}{\lfloor(p-1)/\alpha(p)\rfloor}.
\]
\end{thm}

Unfortunately, unlike Conjecture~A, Conjecture~B has not been established for $q=p$. 


\subsection{Contributions of the Present Paper}\

We will prove the following results.
\begin{itemize}
  \item Conjecture~A is true for $q=p^e$, where $p$ is an odd prime and $\text{gpf}(e)\le p-1$ (Theorem~\ref{T4.1}). This implies that in Theorem~\ref{T1.2}, one can take $p_0(r)=r+1$ (Remark~\ref{R4.3}).
  \item Conjecture~B is true for $q=p^e$, where $e>0$ is arbitrary and $p$ is an odd prime satisfying $\alpha(p)>(p-1)/2$ (Theorem~\ref{T5.2}).
  \item Conjecture~\ref{C1.1} is true (Theorem~\ref{T6.2}).
\end{itemize} 

\begin{rmk}\label{R1.5}\rm 
Although Conjectures~A and B were originally stated for an odd characteristic, their status also appears to be unsettled for $p=2$.
\end{rmk}


\section{Power Sums of $A_k$ and $B_k$}

Hermite's criterion (see Lidl and Niederreiter \cite[Theorem~7.4]{LN}) states that a polynomial $f\in\f_q[X]$ is a PP of $\f_q$ if and only if 
\begin{itemize}
  \item [(i)] $0$ is the only root of $f$ in $\f_q$ , and
  \item [(ii)] $\sum_{x\in\f_q}f(x)^s=0$ for all $1\le s\le q-2$.
\end{itemize}

Let $q$ be any prime power (even or odd). For each integer $a>0$, let $a^*\in\{1,\dots,q-1\}$ be such that $a^*\equiv a\pmod{q-1}$; we also define $0^*=0$. Note that for all $a\ge 0$ and $x\in\f_q$, $x^a=x^{a^*}$. We always assume that $1\le k\le q-1$; additional assumptions on $k$, when they apply, will be included in the context.

\begin{lem}\label{L2.1}
For $1\le s\le q-1$,
\begin{equation}\label{2.1}
\sum_{x\in\f_q}A_k(x)^s=(-1)^{s+1}\sum_{i=0}^s(-1)^i\binom si\binom{(ki)^*}{(2ks)^*}.
\end{equation}
\end{lem}

\begin{proof}
We have 
\begin{align*}
\sum_{x\in\f_q}A_k(x)^s&=\sum_{x\in\f_q^*}x^{ks}\bigl[(x+1)^k-x^k\bigr]^s\\
&=\sum_{x\in\f_q^*}x^{ks}\sum_i\binom si(x+1)^{ki}(-x^k)^{s-i}\\
&=\sum_{x\in\f_q^*}\sum_i(-1)^{s-i}\binom si x^{2ks-ki}(x+1)^{(ki)^*}\\
&=\sum_{x\in\f_q^*}\sum_i(-1)^{s-i}\binom si x^{2ks-ki}\sum_j\binom{(ki)^*}j x^{(ki)^*-j}\\
&=\sum_{i,j}(-1)^{s-i}\binom si\binom{(ki)^*}j\sum_{x\in\f_q^*}x^{2ks-j}\\
&=(-1)^{s+1}\sum_i(-1)^i\binom si\sum_{j\equiv 2ks\,(\text{mod}\,q-1)} \binom{(ki)^*}j.
\end{align*}
If $2ks\not\equiv 0\pmod{q-1}$,
\[
\sum_{x\in\f_q}A_k(x)^s=(-1)^{s+1}\sum_i(-1)^i\binom si\binom{(ki)^*}{(2ks)^*}.
\]
If $2ks\equiv 0\pmod{q-1}$,
\[
\begin{split}
\sum_{x\in\f_q}A_k(x)^s\,&=(-1)^{s+1}\sum_i(-1)^i\binom si\Bigl[\binom{(ki)^*}0+\binom{(ki)^*}{q-1}\Bigr]\cr
&=(-1)^{s+1}\sum_i(-1)^i\binom si\binom{(ki)^*}{(2ks)^*}.
\end{split}
\]
Hence \eqref{2.1} always holds.
\end{proof}

\begin{lem}\label{L2.2}
\begin{itemize}
  \item [(i)] If $q$ is even,
\begin{equation}\label{2.2}
\sum_{x\in\f_q}B_k(x)^s=\sum_{i=0}^s\binom si\binom{(2ki)*}{(ks)^*},\qquad 1\le s\le q-1.
\end{equation}

\item[(ii)] If $q$ is odd,
\begin{equation}\label{2.3}
\sum_{x\in\f_q}B_k(x)^s=-(-2)^s\sum_{i,j}2^{-i}(-1)^j\binom si\binom ij\binom{(2kj)*}{(ki)^*},\qquad 1\le s\le q-1.
\end{equation}   
\end{itemize} 
\end{lem}

\begin{proof}
(i) If $k=q-1$, 
\[
B_k(x)=(x+1)^{2(q-1)}-1=
\begin{cases}
1&\text{if}\ x=1,\cr
0&\text{if}\ x\in\f_q\setminus\{1\},
\end{cases}
\]
so the left side of \eqref{2.2} is $1$. On the other hand, the right side of \eqref{2.2} equals 
\[
\sum_{i=1}^s\binom si=1,
\]
and hence \eqref{2.2} holds.

Now assume that $1\le k< q-1$. The calculation is identical to the proof of Lemma~\ref{L2.1}. We have
\begin{align*}
\sum_{x\in\f_q}B_k(x)^s\,&=\sum_{x\in\f_q^*}\Bigl(\bigl[(x+1)^{2k}+1\bigr]x^{-k}\Bigr)^s\\
&=\sum_{x\in\f_q^*}x^{-ks}\bigl[(x+1)^{2k}+1\bigr]^s\\
&=\sum_{x\in\f_q^*}x^{-ks}\sum_i\binom si(x+1)^{(2ki)^*}\\
&=\sum_{x\in\f_q^*}x^{-ks}\sum_i\binom si\sum_j\binom{(2ki)^*}j x^j\\
&=\sum_{i,j}\binom si\binom{(2ki)^*}j\sum_{x\in\f_q^*}x^{j-ks}\\
&=\sum_i\binom si\sum_{j\equiv ks\,(\text{mod}\,q-1)}\binom{(2ki)^*}j.
\end{align*}
If $ks\not\equiv 0\pmod{q-1}$,
\[
\sum_{x\in\f_q}B_k(x)^s=\sum_i\binom si\binom{(2ki)^*}{(ks)^*}.
\]
If $ks\equiv0\pmod{q-1}$,
\[
\sum_{x\in\f_q}B_k(x)^s=\sum_i\binom si\Bigl[\binom{(2ki)^*}0+\binom{(2ki)^*}{q-1}\Bigr]=\sum_i\binom si\binom{(2ki)^*}{(ks)^*}.
\]

(ii) We have
\begin{align*}
\sum_{x\in\f_q}B_k(x)^s\,&=\sum_{x\in\f_q^*}\Bigl(\bigl[(x+1)^{2k}-1\bigr]x^{-k}-2\Bigr)^s\\
&=\sum_{x\in\f_q^*}\sum_i\binom si\bigl[(x+1)^{2k}-1\bigr]^ix^{-ki}(-2)^{s-i}\\
&=(-2)^s\sum_{x\in\f_q^*}\sum_i(-2)^{-i}\binom six^{-ki}\sum_j\binom ij(x+1)^{(2kj)^*}(-1)^{i-j}\\
&=(-2)^s\sum_{x\in\f_q^*}\sum_{i,j}2^{-i}(-1)^j\binom si\binom ijx^{-ki}\sum_l\binom{(2kj)^*}l x^l\\
&=(-2)^s\sum_{i,j,l}2^{-i}(-1)^j\binom si\binom ij\binom{(2kj)^*}l\sum_{x\in\f_q^*}x^{l-ki}\\
&=-(-2)^s\sum_{i,j}\sum_{l\equiv ki\,(\text{mod}\,q-1)}2^{-i}(-1)^j\binom si\binom ij\binom{(2kj)^*}l.
\end{align*}
Note that if $l\equiv ki\pmod{q-1}$ and $0\le l\le (2kj)^*$, then either $l=(ki)^*$ or $i=0$, $j>0$ and $l=q-1$; in the latter case, $\binom ij=0$. Therefore, we have
\[
\sum_{x\in\f_q}B_k(x)^s=-(-2)^s\sum_{i,j}2^{-i}(-1)^j\binom si\binom ij\binom{(2kj)^*}{(ki)^*}.
\]
\end{proof}

\begin{thm}\label{T2.3}
\begin{itemize}
  \item [(i)] $A_k$ is a PP of $\f_q$ if and only if $\text{\rm gcd}(k,q-1)=1$ and 
\begin{equation}\label{2.4}
\sum_i(-1)^i\binom si\binom{(ki)^*}{(2ks)^*}=0\quad\text{for all}\ 1\le s\le q-2.
\end{equation}

\item[(ii)] $B_k$ is a PP of $\f_q$ if and only if $\text{\rm gcd}(k,q-1)=1$ and
\begin{equation}\label{2.5}
\sum_i(-1)^i\binom si\binom{(2ki)^*}{(ks)^*}=(-2)^s\quad\text{for all}\ 1\le s\le q-2.
\end{equation}
\end{itemize}
\end{thm}

We remind the reader that according to our convention, \eqref{2.4} and \eqref{2.5} are to be treated as equations in characteristic $p$.

\begin{proof}[Proof of Theorem~\ref{T2.3}] We prove the claims using Hermite's criterion.

\medskip
(i)  
Clearly, $0$ is the only root of $A_k$ in $\f_q$ if and only if $\text{gcd}(k,q-1)=1$. By \eqref{2.1}, $\sum_{x\in\f_q}A_k(x)^s=0$ for all $1\le s\le q-2$ if and only if \eqref{2.4} holds.

\medskip
(ii) We consider even and odd $q$'s separately. 

\medskip
{\bf Case 1.} Assume that $q$ is even. We have $B_k=[(X+1)^{2k}-1]X^{q-1-k}$.

If $q=2$, then $k=1$ and $B_k=X^2$, which is a PP of $\f_2$. In this case, \eqref{2.5} is vacuously satisfied.

Now assume that $q>2$. 
Clearly, $0$ is the only root of $B_k$ in $\f_q$ if and only if $\text{gcd}(k,q-1)=1$. By \eqref{2.2}, $\sum_{x\in\f_q}B_k(x)^s=0$ for all $1\le s\le q-2$ if and only if  \eqref{2.5} holds.

\medskip
{\bf Case 2.} Assume that $q$ is odd.

\medskip
$1^\circ$ We claim that if $B_k$ is a PP of $\f_q$, then $\text{gcd}(k,(q-1)/2)=1$. Otherwise, $\text{gcd}(2k,q-1)>2$ and the equation $(x+1)^{2k}-1=0$ has at least two roots $x_1,x_2\in\f_q^*$. Then $B_k(x_1)=-2=B_k(x_2)$, which is a contradiction.

\medskip
$2^\circ$ We claim that $B_k$ is a PP of $\f_q$ if and only if $\text{gcd}(k,(q-1)/2)=1$ and \eqref{2.5} holds. 

By $1^\circ$ and \eqref{2.3}, we only have to show that under the assumption that $\text{gcd}(k,(q-2)/2)=1$, 
\begin{equation}\label{2.6}
\sum_{i,j}2^{-i}(-1)^j\binom si\binom ij\binom{(2kj)^*}{(ki)^*}=
\begin{cases}
0&\text{for}\ 1\le s\le q-2,\cr
1&\text{for}\ s=q-1,
\end{cases}
\end{equation}
if and only if \eqref{2.5} holds. Set
\[
S_i=2^{-i}\sum_j(-1)^j\binom ij\binom{(2kj)^*}{(ki)^*},\qquad 0\le i\le q-1.
\]
Then \eqref{2.6} is equivalent to
\begin{equation}\label{2.7}
\sum_i\binom si S_i=
\begin{cases}
1&\text{if}\ s=0,\cr
0&\text{if}\ 1\le s\le q-2,\cr
1&\text{if}\ s=q-1.
\end{cases}
\end{equation}
Equation~\eqref{2.7} is a recursion for $S_i$, which has a unique solution 
\[
S_i=
\begin{cases}
(-1)^i&\text{if}\ 0\le i\le q-2,\cr
2&\text{if}\ i=q-1.
\end{cases}
\]
Therefore, \eqref{2.6} is equivalent to 
\begin{equation}\label{2.8}
\sum_j(-1)^j\binom ij\binom{(2kj)^*}{(ki)^*}=
\begin{cases}
(-2)^i&\text{if}\ 0\le i\le q-2,\cr
2&\text{if}\ i=q-1.
\end{cases}
\end{equation}
It remains to show that when $i=0$ and $q-1$, \eqref{2.8} is automatically satisfied. When $i=0$, \eqref{2.8} is clearly satisfied. When $i=q-1$, 
\[
\begin{split}
\sum_j(-1)^j\binom ij\binom{(2kj)^*}{(ki)^*}\,&=\sum_{j=\frac{q-1}2,\,q-1}(-1)^j\binom{q-1}j\binom{(2kj)^*}{q-1}\cr
&=(-1)^{\frac{q-1}2}\binom{-1}{\frac{q-1}2}+(-1)^{q-1}\binom{-1}{q-1}=2.
\end{split}
\]

$3^\circ$ To complete the proof of Case 2, it remains to show that if $B_k$ is a PP of $\f_q$, then $\text{gcd}(k,q-1)=1$, that is, $k$ must be odd. This is given by Lemma~\ref{L3.5} later.
\end{proof}

\begin{rmk}\label{R2.4}\rm 
In \eqref{2.5}, we have 
\[
\binom{(2ki)^*}{(ks)^*}=\binom{2ki}{ks}\qquad\text{if}\ 0\le i\le s<\frac{q-1}k.
\]
In fact, if $2ki\le q-1$, then $(2ki)^*=2ki$. If $2ki\ge q$, then $(2ki)^*=2ki-(q-1)<ki\le ks<q-1$, and hence $\binom{(2ki)^*}{(ks)^*}=0$. On the other hand, the sum $ks+(2ki-ks)$ has a carry in base $q$, and hence we also have $\binom{2ki}{ks}=0$.
\end{rmk} 


\section{Facts about $A_k$ and $B_k$}\label{S4}

The permutation properties of $A_k$ and $B_k$ are encoded, respectively, in \eqref{2.4} and \eqref{2.5}, which are not transparent in $q$ and $k$. In this section, we extract from these equations some facts about $A_k$ and $B_k$ that are more explicit in $q$ and $k$. The results collected here include both known and new. Slightly different proofs of the known results are provided for the reader's convenience.

Assume that $q>2$ and $1\le k\le q-1$, and let 
\begin{equation}\label{3.1}
a:=\Bigl\lfloor\frac{q-1}k\Bigr\rfloor.
\end{equation}
When $\text{gcd}(k,q-1)=1$, let $k',b\in\{1,\dots,q-1\}$ be such that 
\begin{equation}\label{3.2}
k'k\equiv 1\pmod{q-1},\qquad bk\equiv-1\pmod {q-1},
\end{equation}
and set
\begin{equation}\label{3.3}
c:=\Bigl\lfloor\frac{q-1}{k'}\Bigr\rfloor.
\end{equation}
Note that 
\begin{equation}\label{3.4}
\frac{q-1}{a+1}<k\le\frac{q-1}a
\end{equation}
and
\begin{equation}\label{3.5}
\frac{q-1}{c+1}< k'\le\frac{q-1}c.
\end{equation}

The following obvious fact will be used frequently.

\begin{fac}\label{F3.1}\rm
$A_k$ is a PP of $\f_q$ if and only if $A_{(pk)^*}$ is. The same is true for $B_k$.
\end{fac}

\begin{lem}\label{L3.2}
If $1<k\le q-1$ and $A_k$ is a PP of $\f_q$, then
\begin{equation}\label{3.6}
\binom{ka}{q-1-ka}\equiv 0\pmod p,
\end{equation}
\begin{equation}\label{3.7}
\binom{2c}c\equiv 0\pmod p.
\end{equation}
\end{lem}

\begin{proof}
$1^\circ$ We first prove \eqref{3.6}. By Theorem~\ref{T2.3} (i), $\text{gcd}(k,q-1)=1$, and hence \eqref{3.4} becomes
\begin{equation}\label{3.8}
\frac{q-1}{a+1}<k<\frac{q-1}a.
\end{equation}
Therefore
\[
q-1<k(a+1)\le 2ka<2(q-1),
\]
which implies that $(2ka)^*=2ka-q+1$. By \eqref{2.4},
\[
\begin{split}
0\,&=\sum_i(-1)^i\binom ai\binom{(ki)^*}{(2ka)^*}=\sum_{2a-\frac{q-1}k\le i\le a}(-1)^i\binom ai\binom{ki}{2ka-q+1}\cr
&=(-1)^a\binom{ka}{2ka-q+1}=(-1)^a\binom{ka}{q-1-ka}.
\end{split}
\]
(Note: in the above, $2a-(q-1)/k\le i\le a$ implies that $i=a$.)

\medskip

$2^\circ$ We now prove \eqref{3.7}. If $c>(q-1)/2$, \eqref{3.7} is automatically satisfied. So we assume that $c\le (q-1)/2$. Since $\text{gcd}(k',q-1)=1$, \eqref{3.5} becomes 
\begin{equation}\label{3.9}
\frac{q-1}{c+1}< k'<\frac{q-1}c.
\end{equation}
If $c=(q-1)/2$, then \eqref{3.9} implies that $k'=1$. It follows that $k=1$, which is a contradiction. Thus $c<(q-1)/2$. Set $s=(cb)^*$. Then
\begin{equation}\label{3.10}
s=q-1-ck',
\end{equation}
and
\begin{equation}\label{3.11}
(2ks)^*=q-1-2c.  
\end{equation}
By \eqref{2.4},
\begin{equation}\label{3.12}
0=\sum_i(-1)^i\binom si\binom{(ki)^*}{(2ks)^*}=\sum_i(-1)^i\binom si\binom{(ki)^*}{q-1-2c}.
\end{equation}
For each $0\le l\le 2c$, let $i(l)\in\{0,\dots,q-1\}$ be such that $(ki(l))^*=q-1-l$. Because of \eqref{3.9}, we have
\[
i(l)=
\begin{cases}
q-1-lk'&\text{if}\ 0\le l\le c,\cr
2(q-1)-lk'&\text{if}\ c+1\le l\le 2c.
\end{cases}
\]
When $0\le l<c$,
\[
i(l)=q-1-lk'>q-1-ck'=s.
\]
When $c<l\le 2c$, we also have
\[
i(l)=2(q-1)-lk'>q-1-ck'=s.
\]
When $l=c$, $i(l)=s$. Therefore \eqref{3.12} becomes
\[
0=(-1)^s\binom{q-1-c}{q-1-2c}=(-1)^s\binom{q-1-c}{c}=(-1)^s\binom{-1-c}c=(-1)^{s+c}\binom{2c}c.
\]
\end{proof}

\begin{cor}\label{C3.3}
Conjecture A is true for $q=p$.
\end{cor}

\begin{proof}
Let $1<k\le p-1$. Since $0\le p-1-ka\le ka\le p-1$, we have
\[
\binom{ka}{p-1-ka}\not\equiv 0\pmod p.
\]
By Lemma~\ref{L3.2}, $A_k$ is not a PP of $\f_q$.
\end{proof}

\begin{rmk}\label{R4.4}\rm 
Equation~\eqref{3.6} is contained in \cite[Theorem~1]{DLW-FFA-2007}, and Corollary~\ref{C3.3} is implied by the proof of \cite[Theorem~1]{DLW-FFA-2007}.
\end{rmk}

\begin{lem}\label{L3.3.1}
Assume that $A_k$ is a PP of $\f_q$. Then all the base $p$ digits of $k'$ are $0$ or $1$.
\end{lem}

\begin{proof}
We only have to consider the case when $k$ is not a power of $p$. By \eqref{3.7}, we have $c>(p-1)/2$. Write $k'=k_0'p^0+\cdots+k_{e-1}'p^{e-1}$, where $0\le k_i'\le p-1$. Since
\[
c\le\frac{q-1}{k'}\le\frac{p^e-1}{k_{e-1}'p^{e-1}},
\]
we have
\[
k_{e-1}'c\le p-\frac 1{p^{e-1}},
\]
and hence $k_{e-1}'c\le p-1$. It follows that $k_{e-1}'\le(p-1)/c<2$. Replacing $k'$ with $(p^{e-1-i}k')^*$ (and $k$ with $(p^{1+i}k)^*$), we also have $k_i'<2$.
\end{proof}

\begin{lem}\label{L3.4}
Assume that $q$ is odd and $B_k$ is a PP of $\f_q$. Then 
\[
(-2)^{k-1}\equiv 1\pmod p.
\]
\end{lem}

\begin{proof}
We first claim that $k\ne q-1$. If, to the contrary, $k=q-1$, since $\text{gcd}(k,(q-1)/2)=1$ (proof of Theorem~\ref{T2.3}, Case~2, $1^\circ$), we must have $q=3$ and $k=2$. But then $B_k=(X+1)^4-1-2X^2\equiv 2X(X+1)\pmod{X^3-X}$, which is not a PP of $\f_3$.

Since $B_k$ is a PP of $\f_q$, $f:=[(X+1)^{2k}-1]/X^k$ is one-to-one on $\f_q^*$. Since $B_k(0)=0$, we have $f(x)\ne 2$ for all $x\in\f_q^*$. Define $f(0)=2$. Then $f:\f_q\to\f_q$ is a bijection with $f(-2)=0$. Thus
\begin{equation}\label{3.13}
-1=\prod_{x\in\f_q\setminus\{-2\}}f(x)=2\prod_{x\in\f_q\setminus\{0,-2\}}\frac{(x+1)^{2k}-1}{x^k}=2^{k+1}\prod_{x\in\f_q\setminus\{\pm1\}}(x^k+1)(x^k-1).
\end{equation}

{\bf Case 1.}
Assume that $k$ is odd. Since $\text{gcd}(k,(q-1)/2)=1$, we have $\text{gcd}(k,q-1)=1$. Then,
\[
\prod_{x\in\f_q\setminus\{\pm1\}}(x^k+1)=\prod_{y\in\f_q\setminus\{0,2\}}y=-\frac 12,
\]
\[
\prod_{x\in\f_q\setminus\{\pm1\}}(x^k-1)=\prod_{y\in\f_q\setminus\{0,-2\}}y=\frac 12.
\]
Therefore \eqref{3.13} gives
\[
-1=2^{k+1}\Bigl(-\frac 12\Bigr)\frac 12,
\]
that is, $2^{k-1}=1$. 

\medskip

{\bf Case 2.} Assume that $k$ is even. Then $(q-1)/2$ is odd and $\text{gcd}(k,q-1)=2$. 

Let $S$ denote the set of nonzero squares in $\f_q$. We have
\begin{equation}\label{3.14}
\prod_{\alpha\in S}(X-\alpha)=X^{(q-1)/2}-1.
\end{equation}
Setting $X=-1$ in \eqref{3.14} gives $\prod_{\alpha\in S}(\alpha+1)=2$, that is,
\begin{equation}\label{3.15}
\prod_{\alpha\in S\setminus\{1\}}(\alpha+1)=1.
\end{equation}
By \eqref{3.14},
\begin{equation}\label{3.16}
\prod_{\alpha\in S\setminus\{1\}}(X+1-\alpha)=\frac{(X+1)^{(q-1)/2}-1}{X}=\sum_{i=1}^{(q-1)/2}\binom{(q-1)/2}iX^{i-1}.
\end{equation}
Setting $X=0$ in \eqref{3.16} gives
\begin{equation}\label{3.17}
\prod_{\alpha\in S\setminus\{1\}}(\alpha-1)=\frac{q-1}2=-\frac 12.
\end{equation}
By \eqref{3.15} and \eqref{3.17}, 
\[
\prod_{x\in\f_q\setminus\{\pm1\}}(x^k+1)=\prod_{x\in\f_q\setminus\{\pm1\}}(x^2+1)=\Bigl(\prod_{\alpha\in S\setminus\{1\}}(\alpha+1)\Bigr)^2=1,
\]
\[
\prod_{x\in\f_q\setminus\{\pm1\}}(x^k-1)=\Bigl(\prod_{\alpha\in S\setminus\{1\}}(\alpha-1)\Bigr)^2=\frac 14.
\]
Thus \eqref{3.13} becomes $2^{k-1}=-1$.
\end{proof}

\begin{lem}\label{L3.5}
Assume that $q$ is odd, $1<k\le q-1$, and $B_k$ is a PP of $\f_q$. Then $k$ is odd, $a$ and $c$ are even, and
\begin{equation}\label{3.18}
2^{k-1}=1,
\end{equation}
\begin{equation}\label{3.19}
\binom a{\frac a2}=(-1)^{\frac a2}2^a,
\end{equation}
\begin{equation}\label{3.20}
\binom{a-1}{\frac a2}\binom{ka}k=(-1)^{\frac a2-1}2^{a-1},
\end{equation}
\begin{equation}\label{3.21}
\binom b{\frac{q-1}2}=(-1)^{b+\frac{q+1}2}2^b,
\end{equation}
\begin{equation}\label{3.22}
\binom{q-1-ck'}{\frac12(q-1-ck')}=(-1)^{\frac c2+\frac{q-1}2}2^{-ck'},
\end{equation}
\begin{equation}\label{3.23}
(-c+1)\binom{q-1-(c-1)k'}{\frac12[q-1-(c-2)k']}=(-1)^{\frac c2+\frac{q-1}2}2^{-(c-1)k'}.
\end{equation}
\end{lem}

\begin{proof}
$1^\circ$ We first show that $k$ is odd. This will imply \eqref{3.18} through Lemma~\ref{L3.4} and also complete the proof of Theorem~\ref{T2.3}, Case 2, Step $3^\circ$. 
Recall from the proof of Theorem~\ref{T2.3}, Case 2, Step $2^\circ$, that $\text{gcd}(k, (q-1)/2)=1$ and \eqref{2.5} holds.

Assume to the contrary that $k$ is even. Equation \eqref{2.5} with $s=(q-1)/2$ gives
\begin{equation}\label{3.24}
\sum_i(-1)^i\binom{\frac{q-1}2}i\binom{(2ki)^*}{q-1}=(-2)^{\frac{q-1}2}.
\end{equation}
Since $\text{gcd}(2k,q-1)=2$, $(q-1)/2$ is odd. In the above,
\[
\binom{\frac{q-1}2}i\binom{(2ki)^*}{q-1}\ne 0
\]
only if $i=(q-1)/2$. Hence \eqref{3.24} gives $2^{(q-1)/2}=1$. So the order of $2$ in $\f_p^*$ is odd. However, by Lemma~\ref{L3.4}, $2^{k-1}=-1$ has order $2$, which is a contradiction.

\medskip
$2^\circ$ We now prove that $a$ is even and \eqref{3.19} and \eqref{3.20} hold. Since $\text{gcd}(k,(q-1)/2)=1$ and $k$ is odd, we have $\text{gcd}(k,q-1)=1$. Thus \eqref{3.4} becomes
\[
\frac{q-1}{a+1}<k<\frac{q-1}a.
\]
By \eqref{2.5},
\begin{equation}\label{3.25}
\sum_i(-1)^i\binom ai\binom{(2ki)^*}{ka}=(-2)^a.
\end{equation}
In the above, $(2ki)^*\ge ka$ only when $i\ge a/2$. When $a/2<i\le a$, $(2ki)^*=2ki-(q-1)<ka$. Therefore, $a$ must be even and \eqref{3.25} becomes
\[
(-1)^{\frac a2}\binom a{\frac a2}=(-2)^a,
\] 
which is \eqref{3.19}. Also by \eqref{2.5},
\begin{equation}\label{3.26}
\sum_i(-1)^i\binom{a-1}i\binom{(2ki)^*}{k(a-1)}=(-2)^{a-1}.
\end{equation}
In the above, $(2ki)^*\ge k(a-1)$ only when $i\ge(a-1)/2$, i.e., $i\ge a/2$ (since $a$ is even). When $a/2<i\le a-1$, $(2ki)^*=2ki-(q-1)<k(a-1)$. Hence \eqref{3.26} becomes
\[
(-1)^{\frac a2}\binom{a-1}{\frac a2}\binom{ka}{k(a-1)}=(-2)^{a-1},
\]
which is \eqref{3.20}.

\medskip
$3^\circ$ Next, we prove \eqref{3.21}. 
By \eqref{2.5},
\begin{equation}\label{3.27}
(-2)^b=\sum_i(-1)^i\binom bi\binom{(2ki)^*}{(kb)^*}=\sum_i(-1)^i\binom bi\binom{(2ki)^*}{q-2}.
\end{equation}
In the above,
\[
\binom bi\binom{(2ki)^*}{q-2}\ne 0
\]
only if $i=(q-1)/2$. Hence \eqref{3.27} gives
\[
(-2)^b=(-1)^{\frac{q-1}2}\binom b{\frac{q-1}2}\binom{q-1}{q-2}=(-1)^{\frac{q+1}2}\binom b{\frac{q-1}2},
\]
which is \eqref{3.21}.

\medskip
$4^\circ$ Finally, we prove that $c$ is even and \eqref{3.22} and \eqref{3.2} hold.

In \eqref{3.5}, if $k'=(q-1)/c$, since $\text{gcd}(k',q-1)=1$, we must have $k'=1$. Then $k=1$, which is a contradiction. Therefore \eqref{3.5} becomes
\begin{equation}\label{3.28}
\frac{q-1}{c+1}<k'<\frac{q-1}c.
\end{equation}
Let $s=(cb)^*$. Then we have
\begin{gather*}
s=q-1-ck',\\
(ks)^*=q-1-c.
\end{gather*}
By \eqref{2.5},
\begin{equation}\label{3.29}
\sum_i(-1)^i\binom si\binom{(2ki)^*}{q-1-c}=(-2)^s.
\end{equation}
For $0\le l\le c/2$, let $i\in\{0,\dots,q-1\}$ be such that $(2ki)^*=q-1-2l$. By \eqref{3.28},
\[
i=q-1-lk'\quad\text{or}\quad \frac 12(q-1)-lk'.
\]

If $i=q-1-lk'$, then $i>q-1-ck'=s$. 
If $i=\frac 12(q-1)-lk'$, then $i\le s$ only if $l=c/2$. In fact, $\frac 12(q-1)-lk'=i\le s=q-1-ck'$ implies that
\[
k'\le\frac{q-1}{2(c-l)},
\]
which, by \eqref{3.28}, implies that $2(c-l)\le c$, i.e., $l\ge c/2$.

Therefore, the $i$th term of the sum in \eqref{3.29} is nonzero only if $i=\frac 12(q-1)-\frac c2k'$. Hence $c$ must be even and \eqref{3.29} gives 
\[
(-1)^{\frac 12(q-1-ck')}\binom{q-1-ck'}{\frac 12(q-1-ck')}=(-2)^{-ck'},
\]
which is \eqref{3.22}.

To prove \eqref{3.23}, we choose $s=((c-1)b)^*$. We have
\begin{gather*}
s=q-1-(c-1)k',\\
(ks)^*=q-1-(c-1),
\end{gather*}
and \eqref{2.5} gives
\begin{equation}\label{3.30}
\sum_i(-1)^i\binom si\binom{(2ki)^*}{q-1-(c-1)}=(-2)^s.
\end{equation}
For $0\le l\le c/2-1$, let $i\in\{0,\dots,q-1\}$ be such that $(2ki)^*=q-1-2l$. Then
\[
i=q-1-lk'\quad\text{or}\quad \frac 12(q-1)-lk'.
\]

If $i=q-1-lk'$, then $i>s$. If $i=\frac 12(q-1)-lk'$, then $i\le s$ only if $l=c/2-1$. In fact, $i\le s$ implies that
\[
k'\le\frac{q-1}{2(c-1-l)},
\]
which further implies that $2(c-1-l)\le c$, i.e., $l\ge c/2-1$. Therefore, the $i$th term of the sum in \eqref{3.30} is nonzero only if $i=\frac 12[q-1-(c-2)k']$. Hence \eqref{3.30} gives 
\[
\begin{split}
-2^{-(c-1)k'}\,&=(-1)^{\frac c2-1+\frac{q-1}2}\binom{q-1-(c-1)k'}{\frac 12[q-1-(c-2)k']}\binom{q-1-2(\frac c2-1)}{q-1-(c-1)}\cr
&=(-1)^{\frac c2-1+\frac{q-1}2}\binom{q-1-(c-1)k'}{\frac 12[q-1-(c-2)k']}(-c+1),
\end{split}
\]
which is \eqref{3.23}.
\end{proof}

For each odd prime $p$, let
\begin{equation}\label{3.31}
\alpha(p)=\min\Bigl\{u: \text{$u$ is a positive even integer},\ \binom u{u/2}\equiv(-1)^{\frac u2}2^u\pmod p\Bigr\}.
\end{equation}

\begin{rmk}\label{R3.6}\rm 
Since
\[
\binom{p-1}{\frac{p-1}2}\equiv(-1)^{\frac{p-1}2}\pmod p,
\]
we always have $\alpha(p)\le p-1$. 
\end{rmk}

\begin{lem}\label{L3.7}
Assume that $q$ is odd and $1<k\le q-1$. If $B_k$ is a PP of $\f_q$, then all the base $p$ digits of $k$ are $\le(p-1)/\alpha(p)$.
\end{lem}

\begin{proof}
By \eqref{3.19}, $a=\lfloor(q-1)/k\rfloor\ge\alpha(p)$. Let $q=p^e$ and write $k=k_0p^0+\cdots+k_{e-1}p^{e-1}$, where $0\le k_i\le p-1$. We first show that $k_{e-1}\le(p-1)/\alpha(p)$. Assume that $k_{e-1}>0$. Since
\[
a\le \frac{q-1}k\le \frac{p^e-1}{k_{e-1}p^{e-1}},
\]
we have 
\[
k_{e-1}a\le p-\frac1{p^{e-1}}.
\]
Thus $k_{e-1}a\le p-1$, and hence $k_{e-1}\le(p-1)/a\le(p-1)/\alpha(p)$.

Replacing $k$ with $(p^{e-1-i}k)^*$, we conclude that $k_i\le(p-1)/\alpha(p)$.
\end{proof}

We include a quick proof for Theorem~\ref{T1.4}.

\begin{proof}[Proof of Theorem~\ref{T1.4}] Let $q=p^e$.
Assume that $1<k\le q^m-1$ and $B_k$ is a PP of $\f_{q^m}$. Write $k=k_0q^0+\cdots+k_{m-1}q^{m-1}$, $0\le k_i\le q-1$. By Lemma~\ref{L3.7}, all the base $p$ digits of $k$ are $\le \lfloor(p-1)/\alpha(p)\rfloor$. Hence
\[
k_i\le\Bigl\lfloor\frac{p-1}{\alpha(p)}\Bigr\rfloor\frac{q-1}{p-1},\qquad 0\le i\le m-1.
\]
Since Conjecture~B is assumed to be true for $q$, by Fact~\ref{F3.1}, we may assume that $k\equiv 1\pmod{q-1}$, that is,
\[
k_0+\cdots+k_{m-1}\equiv 1\pmod{q-1}.
\]
However, 
\[
k_0+\cdots+k_{m-1}\le m\Bigl\lfloor\frac{p-1}{\alpha(p)}\Bigr\rfloor\frac{q-1}{p-1}\le q-1.
\]
So we must have $k_0+\cdots+k_{m-1}=1$.
\end{proof}

\section{A Theorem on Conjecture~A}

\begin{thm}\label{T4.1}
Conjecture~A is true for $q=p^e$, where $p$ is an odd prime and $\text{\rm gpf}(e)\le p-1$.
\end{thm}

Theorem~\ref{T4.1} is an immediate consequence of Corollary~\ref{C3.3} and the following lemma.

\begin{lem}\label{L4.2}
Let $q$ be a power of an odd prime $p$ and $1\le m\le p-1$. If Conjecture~A is true for $q$, it is also true for $q^m$.
\end{lem}

\begin{proof}
Assume that $A_k$ is a PP of $\f_{q^m}$, where $1\le k\le q^m-2$. Let $k'\in\{1,\dots,q^m-2\}$ be such that $k'k\equiv 1\pmod{q^m-1}$. It suffices to show that $k'$ is a power of $p$. Write $k'=k_0'q^0+\cdots+k_{m-1}'q^{m-1}$, $0\le k_i'\le q-1$. Since $A_k$ is a PP of $\f_q$ and since Conjecture~A is true for $q$, we may assume that $k'\equiv 1\pmod{q-1}$, that is, 
\begin{equation}\label{4.1}
k_0'+\cdots+k_{m-1}'\equiv 1\pmod{q-1}.
\end{equation}
On the other hand, by Lemma~\ref{L3.3.1}, all base $p$ digits of $k'$ are $\le 1$. Hence
\[
k_i'\le\frac{q-1}{p-1},\qquad 0\le i\le m-1.
\]
Therefore,
\begin{equation}\label{4.2}
k_0'+\cdots+k_{m-1}'\le\frac{q-1}{p-1}m\le q-1.
\end{equation}
Combining \eqref{4.1} and \eqref{4.2} gives $k_0'+\cdots+k_{m-1}'=1$.
\end{proof}

\begin{rmk}\label{R4.3}\rm 
In \cite{Kronenthal-FFA-2012}, the author commented that an avenue to improve Theorem~\ref{T1.2} is to find a more explicit form for the function $p_0$ in that theorem. By Theorem~\ref{T4.1}, one can choose $p_0(r)=r+1$.
\end{rmk} 

\section{Conjecture~B with $\alpha(p)>(p-1)/2$}

Our proof of Conjecture~B under the condition $\alpha(p)>(p-1)/2$ follows a simple line of logic. Assume to the contrary that $B_k$ is a PP of $\f_{p^e}$ for some $k\in\{1,\dots,p^e-1\}$ which is not a power of $p$. Then with the help of Lemma~\ref{L5.1}, $a:=\lfloor(p^e-1)/k\rfloor\equiv 0\pmod p$. However, \eqref{3.20} dictates that $a\not\equiv 0\pmod p$, hence a contradiction. 

\begin{lem}\label{L5.1}
Let $p\ge 3$ be a prime.
Let $i$, $j$, $e$ be integers such that $0<i<j\le e-1$, and let
\begin{align*}
k&=k_0p^0+\cdots+k_{i-1}p^{i-1}+p^i+p^j,\quad k_0,\dots,k_{i-1}\in\{0,\dots,p-1\},\\
a&=\Bigl\lfloor\frac{p^e-1}k\Bigr\rfloor,\\
u&=\Bigl\lfloor\frac{e-j}{j-i}\Bigr\rfloor.
\end{align*}
Assume that $a$ is even and 
\begin{equation}\label{5.1}
\frac{p^e-1}{p^i+p^j}-\frac{p^e-1}k\le 1.
\end{equation}
Then 
\begin{equation}\label{5.2}
a=
\begin{cases}
p^{e-j}\bigl[1-p^{i-j}+\cdots+(-1)^{u(i-j)}\big]&\text{if $u$ is odd},\vspace{1mm}\cr
p^{e-j}\bigl[1-p^{i-j}+\cdots+(-1)^{u(i-j)}\big]-1&\text{if $u$ is even}.
\end{cases}
\end{equation}
\end{lem}

\begin{proof}
Write $e-j=u(i-j)+r$, $0\le r<j-i$. We have
\begin{align*}
\frac{p^e-1}{p^i+p^j}=\,&p^{e-j}\frac 1{1+p^{i-j}}-\frac 1{p^i+p^j}\\
=\,&p^{e-j}\bigl[1-p^{i-j}+p^{2(i-j)}-\cdots\bigr]-\frac 1{p^i+p^j}\\
=\,&p^{e-j}\bigl[1-p^{i-j}+\cdots+(-1)^u p^{u(i-j)}\bigr]\\
&+(-1)^{u+1}p^{r+i-j}\bigl[1-p^{i-j}+p^{2(i-j)}-\cdots\bigr]-\frac 1{p^i+p^j}\\
=\,&p^{e-j}\bigl[1-p^{i-j}+\cdots+(-1)^u p^{u(i-j)}\bigr]+(-1)^{u+1}p^{r+i-j}\frac 1{1+p^{i-j}}-\frac 1{p^i+p^j}\\
=\,&p^{e-j}\bigl[1-p^{i-j}+\cdots+(-1)^u p^{u(i-j)}\bigr]+\frac 1{p^i+p^j}\bigl[(-1)^{u+1}p^{r+i}-1\bigr].
\end{align*}
Since $r+i<j$, we have
\begin{align*}
0<\frac 1{p^i+p^j}\bigl[(-1)^{u+1}p^{r+i}-1\bigr]<1\quad&\text{if $u$ is odd},\\
-1<\frac 1{p^i+p^j}\bigl[(-1)^{u+1}p^{r+i}-1\bigr]<0\quad&\text{if $u$ is even}.
\end{align*}
Thus
\begin{equation}\label{5.3}
\Bigl\lfloor\frac{p^e-1}{p^i+p^j}\Bigr\rfloor=
\begin{cases}
p^{e-j}\bigl[1-p^{i-j}+\cdots+(-1)^u p^{u(i-j)}\bigr]&\text{if $u$ is odd},\vspace{1mm}\cr
p^{e-j}\bigl[1-p^{i-j}+\cdots+(-1)^u p^{u(i-j)}\bigr]-1&\text{if $u$ is even}.
\end{cases}
\end{equation}
Note that the right side of \eqref{5.3} is always even. Then \eqref{5.2} follows from \eqref{5.1}, \eqref{5.3} and the assumption that $a$ is even.
\end{proof}

\begin{thm}\label{T5.2}
Conjecture~B is true for $q=p^e$, where $p$ is an odd prime such that $\alpha(p)>(p-1)/2$.
\end{thm}

\begin{proof}Assume to the contrary that there exists $k\in\{1,\dots,p^e-1\}$, which is not a power of $p$, such that $B_k$ is a PP of $\f_{p^e}$. Write
\[
k=k_0p^0+\cdots+k_{e-1}p^{e-1},\quad 0\le k_i\le p-1.
\]
Since $\alpha(p)>(p-1)/2$, by Lemma~\ref{L3.7}, $k_i\le 1$ for all $i$. Let 
\[
a=\Bigl\lfloor\frac{p^e-1}k\Bigr\rfloor.
\]
By Lemma~\ref{L3.5}, $a$ is even, and by \eqref{3.20},
\[
\binom{ka}k\not\equiv 0\pmod p.
\]
In particular, $a\not\equiv 0\pmod p$.

Let $d$ be the distance in $\Bbb Z/e\Bbb Z$ defined by
\[
d([x],[y])=\min\{|x-y|,\,e-|x-y|\},\quad x,y\in\{0,\dots,e-1\}.
\]
This is the arc distance with $[0],\dots,[e-1]$ evenly placed on a circle in that order. Let $l$ be the shortest distance between two indices $i,j\in\Bbb Z/e\Bbb Z$ with $k_i=k_j=1$. Then $1\le l<e$. The $1$'s among $k_0,\dots,k_{e-1}$ cannot be evenly spaced. Otherwise, $\text{gcd}(k,p^e-1)=(p^e-1)/(p^l-1)>1$, which is a contradiction. Therefore, we may write
\[
(k_0,\dots,k_{e-1})=(\overset 0*\;\cdots\;*\;\overset i1\;\underbrace{0\;\cdots\;0}_{l-1}\;\overset j1\;\underbrace{0\;\cdots\;\overset{e-1}0}_l),
\]
where $j=e-1-l$, $i=j-l=e-1-2l$. We have
\[
u=\Bigl\lfloor\frac{e-j}{j-i}\Bigr\rfloor= \Bigl\lfloor\frac{l+1}l\Bigr\rfloor=
\begin{cases}
2&\text{if}\ l=1,\cr
1&\text{if}\ l\ge 2.
\end{cases}
\]

{\bf Case 1.}
Assume that $l=1$. Since 
\[
k_0p^0+\cdots+k_ip^i\le p^0+\cdots+p^i=\frac{p^j-1}{p-1}<\frac{p^j}{p-1},
\]
we have
\begin{align*}
\frac{p^e-1}{p^i+p^j}-\frac{p^e-1}k&=(p^e-1)\Bigl[\frac 1{p^i+p^j}-\frac 1{k_0p^0+\cdots+k_ip^i+p^j}\Bigr]\\
&<(p^e-1)\biggl[\frac 1{p^i+p^j}-\frac 1{\frac{p^j}{p-1}+p^j}\biggr]\\
&=(p^e-1)\frac 1{p^j}\Bigl[\frac p{p+1}-\frac{p-1}p\Bigr]\\
&=\frac{p^e-1}{p^jp(p+1)}<p^{e-j-2}=1.
\end{align*}
Thus by \eqref{5.2},
\[
a=p^{e-j}\bigl[1-p^{i-j}+\cdots+(-1)^up^{u(i-j)}\bigr]-1=p^2(1-p^{-1})\equiv 0\pmod p,
\]
which is a contradiction.

\medskip
{\bf Case 2.} Assume that $l\ge 2$. Since the distance between the indices of any two consecutive $1$'s among $k_0,\dots,k_{e-1}$ is $\ge l$, we have
\[
k_0p^0+\cdots+k_ip^i<p^i+p^{i-l}+p^{i-2l}+\cdots=p^i\frac{p^l}{p^l-1}.
\]
Hence
\begin{align*}
\frac{p^e-1}{p^i+p^j}-\frac{p^e-1}k&=(p^e-1)\Bigl[\frac 1{p^i+p^j}-\frac 1{k_0p^0+\cdots+k_ip^i+p^j}\Bigr]\\
&<(p^e-1)\biggl[\frac 1{p^i+p^j}-\frac 1{\frac{p^l}{p^l-1}p^i+p^j}\biggr]\\
&=\frac{p^e-1}{p^l-1}\frac{p^i}{(p^i+p^j)\bigl(\frac{p^l}{p^l-1}p^i+p^j\bigr)}\\
&<\frac{p^{e+i}}{p(p^i+p^j)\bigl(\frac{p^l}{p^l-1}p^i+p^j\bigr)}\\
&<\frac{p^{e+i}}{p^{2j+1}}=p^{e+e-1-2l-2(e-1-l)-1}=1.
\end{align*}
Therefore by \eqref{5.2},
\[
a=p^{e-j}\bigl[1-p^{i-j}+\cdots+(-1)^up^{u(i-j)}\bigr]=p^{l+1}(1-p^{-l})\equiv 0\pmod p,
\]
which is a contradiction.
\end{proof}

Many odd primes $p$ satisfy the condition $\alpha(p)>(p-1)/2$. Among the first $1000$ odd primes $p$, the equation $\alpha(p)=p-1$ holds with $211$ exceptions. The first few exceptions are $\alpha(29)=10$, $\alpha(31)=8$, $\alpha(47)=18, \dots$. In fact, for any odd prime $p$, either $\alpha(p)=p-1$ or $\alpha(p)\le(p-1)/2$; this follows from a symmetry described below.

Note that for integer $m\ge 0$,
\[
2^{-2m}\binom{2m}m=\frac{(2m)!}{(2^m\cdot m!)^2}=\frac{(2m-1)!!}{(2m)!!}.
\]
Thus $\alpha(p)$ is the smallest positive even integer $2m$ ($\le p-1$) such that 
\begin{equation}\label{5.4}
\frac{(2m-1)!!}{(2m)!!}\equiv(-1)^m\pmod p.
\end{equation}
Let $0\le m\le(p-1)/2$. Since
\[
\begin{split}
&2^{-2m}\binom{2m}m\Bigl[2^{-(p-1-2m)}\binom{p-1-2m}{\frac{p-1}2-m}\Bigr]^{-1}\cr
=\,&\frac{(2m-1)!!}{(2m)!!}\cdot\frac{(p-1-2m)!!}{(p-2-2m)!!}=\frac{(p-2)!!}{(p-1)!!}=\prod_{\substack{2\le i\le p-1\cr i\,\text{even}}}\frac{p-i}i=(-1)^{\frac{p-1}2},
\end{split}
\]
condition ~\eqref{5.4} is symmetric for $m$ and $(p-1)/2-m$.

For integers $i\le j$, denote $i(i+1)\cdots j$ by $[i,j]$. Then we have
\begin{gather*}
\binom{2m}m=\frac{[m+1,2m]}{[1,m]},\\
\binom{p-1-2m}{\frac{p-1}2-m}=\frac{[\frac{p+1}2-m, p-1-2m]}{[1,\frac{p-1}2-m]}=\frac{[2m+1,\frac{p-1}2+m]}{[\frac{p+1}2+m,p-1]}.
\end{gather*}
Hence 
\[
\begin{split}
\binom{2m}m\binom{p-1-2m}{\frac{p-1}2-m}\,&=\frac{[m+1,m+\frac{p-1}2]}{[1,m][m+\frac{p+1}2,p-1]}=\frac{[m+1,m+\frac{p-1}2]^2}{[1,p-1]}\cr
&=-\Bigl[m+1,m+\frac{p-1}2\Bigr]^2.
\end{split}
\]
Therefore, if \eqref{5.4} is satisfied, one has
\[
\prod_{i=1}^{\frac{p-1}2}(m+i)^2\equiv(-1)^{\frac{p+1}2}\pmod p.
\]


\section{Proof of Conjecture~\ref{C1.1}}

We continue to use the notation introduced at the beginning of Section~\ref{S4}. For $1\le k\le q-1$ with $\text{gcd}(k,q-1)=1$, the parameters $k'$, $b$ and $c$ are defined in \eqref{3.2} and \eqref{3.3}.

Assume to the contrary that Conjecture~\ref{C1.1} is false. Then for some $k\in\{1,\dots,q-1\}$ which is not a power of $p$, both $A_k$ and $B_k$ are PPs of $\f_q$. We will see that the same argument as in the proof of Theorem~\ref{T5.2} gives that $c:=\lfloor(q-1)/k'\rfloor\equiv 0\pmod p$. The purpose of the following lemma is to establish an equation that cannot be satisfied when $c\equiv 0\pmod p$.

\begin{lem}\label{L6.1}
Assume that $q$ is odd, $1< k\le q-1$, and both $A_k$ and $B_k$ are PPs of $\f_q$. Then $c$ is even and
\begin{equation}\label{6.1}
2^{-2ck'}=\binom{2(q-1)-2ck'}{q-1-ck'}+(-1)^{\frac{q-1}2+\frac c2+1}\binom{2(q-1)-2ck'}{\frac 12(q-1)-(\frac c2-1)k'}\binom{2c}{c+2}.
\end{equation}
\end{lem}

\begin{proof}
By Lemma~\ref{L3.5}, $c$ is even. Let $s=(2cb)^*$. Since $2cb\not\equiv 0\pmod{q-1}$, we have $1\le s\le q-2$. Clearly, $2ck'>q-1$. (Otherwise, $2c\le (q-1)/k'$, which implies that $2c\le c$, a contradiction.) It follows that
\[
s=2(q-1)-2ck'.
\]
Note that $c<(q-1)/2$. (Otherwise, since $\text{gcd}(k',q-1)=1$, we have $k'<(q-1)/2\le 2$, which implies that $k'=1$, i.e., $k=1$, which is a contradiction.) Thus
\[
(ks)^*=q-1-2c.
\]
By \eqref{2.5}, 
\begin{equation}\label{6.2}
\sum_i(-1)^i\binom si\binom{(2ki)^*}{q-1-2c}=(-2)^s.
\end{equation}
For each $0\le l\le c$, let $i\in\{0,\dots,q-1\}$ be such that $(2ki)^*=q-1-2l$. Then 
\[
i=\frac32(q-1)-lk'\quad\text{or}\quad q-1-lk'\quad\text{or}\quad \frac 12(q-1)-lk'.
\]
In each of these cases, we determine the necessary conditions on $l$ such that $i$ satisfies $0\le i\le s$.

\medskip
{\bf Case 1.} Assume that $i=\frac32(q-1)-lk'$. In this case, 
\[
\begin{split}
i\,&\ge\frac 32(q-1)-ck'>2(q-1)-2ck'\qquad\text{(since $2ck'>q-1$)}\cr
&=s.
\end{split}
\]

{\bf Case 2.} Assume that $i=q-1-lk'$. In this case we always have $i\ge 0$. Moreover, 
\[
\begin{split}
i\le s&\Leftrightarrow  q-1-lk'\le 2(q-1)-2ck'\cr
&\Leftrightarrow l\ge 2c-\frac{q-1}{k'}\cr
&\Leftrightarrow  l\ge c.
\end{split}
\]

{\bf Case 3.} Assume that $i=\frac 12(q-1)-lk'$. In this case, $i\ge 0$ if and only if $l\le c/2$. Moreover,
\[
\begin{split}
i\le s&\Leftrightarrow  \frac 12(q-1)-lk'\le 2(q-1)-2ck'\cr
&\Leftrightarrow l\ge 2c-\frac 32\cdot\frac{q-1}{k'}\cr
&\Rightarrow  l>2c-\frac 32(c+1)\cr
&\Rightarrow  l\ge\frac c2-1.
\end{split}
\]
Combining the above three cases, we see that \eqref{6.2} becomes
\begin{equation}\label{6.3}
\begin{split}
2^{-2ck'}
=\,&\binom{2(q-1)-2ck'}{q-1-ck'}\cr
&+(-1)^{\frac{q-1}2+\frac c2+1}\binom{2(q-1)-2ck'}{\frac 12(q-1)-(\frac c2-1)k'}\binom{q-1-2(\frac c2-1)}{q-1-2c}\cr
&+(-1)^{\frac{q-1}2+\frac c2}\binom{2(q-1)-2ck'}{\frac 12(q-1)-\frac c2k'}\binom{q-1-c}{q-1-2c}.
\end{split}
\end{equation}
In the above,
\[
\binom{q-1-2(\frac c2-1)}{q-1-2c}=\binom{-c+1}{2+c}=\binom{2c}{c+2},
\]
and, by \eqref{3.7},
\[
\binom{q-1-c}{q-1-2c}=\binom{-1-c}c=\binom{2c}c=0.
\]
Hence \eqref{6.1} follows from \eqref{6.3}.
\end{proof}

\begin{thm}\label{T6.2}
Conjecture~\ref{C1.1} is true.
\end{thm}

\begin{proof} Assume to the contrary that Conjecture~\ref{C1.1} is false. Then there exists $1\le k\le q-1$, which is not a power of $p$, such that both $A_k$ and $B_k$ are PPs of $\f_q$.

By Lemma~\ref{L3.3.1}, all the base $p$ digits of $k'$ are $\le 1$. By exactly the same argument as in the proof of Theorem~\ref{T5.2}, with $k$ and $a$ replaced by $k'$ and $c$, respectively, we conclude that we may assume that $c\equiv 0\pmod p$. Then obviously, 
\begin{equation}\label{6.4}
\binom{2c}{c+2}=0.
\end{equation}
Since $q-1-ck'\equiv p-1\pmod p$, the sum $(q-1-ck')+(q-1-ck')$ has a carry in base $p$ at $p^0$, implying that 
\begin{equation}\label{6.5}
\binom{2(q-1)-2ck'}{q-1-ck'}=0.
\end{equation}
Combining \eqref{6.1}, \eqref{6.4} and \eqref{6.5}, we have a contradiction.
\end{proof}

As a concluding remark, we reiterate that Conjectures~A and B are still open and we hope that they will stimulate further research.

\section*{Acknowledgments}
The authors are thankful to Qing Xiang who facilitated their collaboration.



\begin{thebibliography}{99}


\bibitem{Bollobas98} B. Bollob\' as, {\it Modern Graph Theory},
Springer-Verlag, New York, 1998.

\bibitem{Bondy02} J. A. Bondy, {\it Extremal problems of Paul Erd\" os on circuits in graphs},
Paul Erd\" os and His Mathematics. II, Bolyai Society, Mathematical
Studies, 11, Budapest, 2002, 135 -- 156.


\bibitem{Chandler05} D. B. Chandler, {\it Personal communication}, August 2005.

\bibitem{Che99} W. E. Cherowitzo,
{\it Hyperovals in Desarguesian planes: an electronic update}, Informal notes, \hfil\break {\tt
http://www-math.cudenver.edu/\~{}wcherowi/res.html}, February 2000.

\bibitem{Dmy04} V. Dmytrenko, {\it Classes of Polynomial Graphs}, Ph.D.
Thesis, University of Delaware, 2004.


\bibitem
{DLW-FFA-2007}
V. Dmytrenko, F. Lazebnik, J. Williford, {\it On monomial graphs of girth eight}, Finite Fields Appl. {\bf 13} (2007), 828 -- 842.


\bibitem{ExoJaj13} G. Exoo and R. Jajcay, {\it Dynamic cage survey}, The Electronic Journal of Combinatorics (2013), $\#$DS16, 1 -- 55.
    
\bibitem{FurSim13} Z. F\"uredi and M. Simonovits, {\it The history of degenerate (bipartite) extremal graph problems},
 Erd\H{o}s centennial, 169 -- 264, Bolyai Soc. Math. Stud., 25, J\'{a}nos Bolyai Math. Soc., Budapest, 2013.


\bibitem{Glynn83}
D. G. Glynn, {\it Two new sequences of ovals in finite Desarguesian
planes of even order},  Combinatorial Mathematics, X (Adelaide,
1982),  217 -- 229, Lecture Notes in Math. 1036, Springer, Berlin,
1983.

\bibitem{Glynn89}
D. G. Glynn, {\it A condition for the existence of ovals in ${\rm
PG}(2,q), q$ even}, Geom. Dedicata {\bf 32} (1989), 247 -- 252.

\bibitem{Hor02} S. Hoory, {\it The size of bipartite graphs with a given girth}, 
J. Combin. Theory Ser. B {\bf 86} (2002), 215 -- 220. 

\bibitem{Hou-FFA-2015}
X. Hou, {\it Permutation polynomials over finite fields --- a survey of recent advances}, Finite Fields Appl. {\bf 32} (2015), 82 -- 119.

\bibitem
{Kronenthal-FFA-2012}
B. G. Kronenthal, {\it Monomial graphs and generalized quadrangles}, Finite Fields Appl. {\bf 18} (2012), 674 -- 684.


\bibitem{LU93}
F. Lazebnik and V. A. Ustimenko, {\it New examples of graphs without small
cycles and of large size}, European J. Combin. {\bf 14} (1993), 
445 -- 460.


\bibitem{LWol01}
F. Lazebnik and  A. J. Woldar, {\it General properties of some
families of graphs defined by systems of equations},
J. Graph Theory {\bf 38} (2001), 65 -- 86.


\bibitem{LN} R. Lidl and H. Niederreiter, {\it Finite Fields},
Encyclopedia of
Mathematics and Its Applications 20, Cambridge University Press,
Cambridge, 1997.

\bibitem{Lucas-AJM-1878}
E. Lucas, {\it Th\'eorie des fonctions num\'eriques simplement p\'eriodiques}, Amer. J.  Math. {\bf 1} (1878), 197 -- 240.

\bibitem{MilSir13} M. Miller and J. \v{S}ir\'{a}\v{n}, {\it Moore graphs and beyond:
A survey of the degree/diameter problem}, The Electronic Journal of Combinatorics (2013), {\bf 20}\,(2), $\#$DS14v2, 1 -- 92.

\bibitem{Payne90} S. E. Payne, {\it A census of finite generalized quadrangles},  W. M. Kantor, R. A. Liebler, S. E. Payne, E. E. Shult (Editors), Finite Geometries, Buildings, and Related Topics, Clarendon, Oxford, 1990, 29 -- 36.


\bibitem{Payne-Thas2009} S. E. Payne and J. A. Thas, {\it Finite Generalized Quadrangles}, Ems Series of Lectures in Mathematics,  European Mathematical Society, 2nd edition, 2009.  
    

\bibitem{P70} S. E. Payne, {\it Affine representation of generalized quadrangles},
 Journal of Algebra {\bf 51} (1970), 473 -- 485.



\bibitem{vM98} H. van Maldeghem, {\it Generalized Polygons},
Birkh\" auser, Basel-Boston-Berlin, 1998.

\bibitem{Vig02} R. Viglione,  {\it Properties of Some Algebraically Defined Graphs}, Ph.D.
Thesis, University of Delaware, 2002.

\end{thebibliography}
\end{document}